\documentclass[12pt,reqno]{amsart}

\usepackage{epsfig}
\usepackage{amsmath, amsthm, amsfonts, amssymb, mathrsfs}
\usepackage{tikz-cd}
\usepackage{graphicx}
\usepackage{comment}
\usepackage[utf8]{inputenc}

\numberwithin{equation}{section}

\DeclareMathOperator\ord{ord}

\newcommand{\C}{{\mathbb C}}
\newcommand{\N}{{\mathbb N}}
\newcommand{\Q}{{\mathbb Q}}
\newcommand{\Z}{{\mathbb Z}}

\newtheorem{theo}{{\sc \bf Theorem}}[section]

\newtheorem{lem}[theo]{{\sc \bf Lemma}}
\newtheorem{prop}[theo]{{\sc \bf Proposition}}

\newenvironment{rem}{\medskip\noindent{\bf Remark:\/} }{\medskip}
\newenvironment{defin}{\medskip\noindent{\bf Definition:\/} }{\medskip}

\begin{document}

\title{Crossed Product C$^*$-Algebras Associated with p-Adic Multiplication}

\author[Hebert]{Shelley Hebert}
\address{Department of Mathematics and Statistics,
Mississippi State University,
175 President's Cir. Mississippi State, MS 39762, U.S.A.}
\email{sdh7@msstate.edu}

\author[Klimek]{Slawomir Klimek}
\address{Department of Mathematical Sciences,
Indiana University-Purdue University Indianapolis,
402 N. Blackford St., Indianapolis, IN 46202, U.S.A.}
\email{sklimek@math.iupui.edu}

\author[McBride]{Matt McBride}
\address{Department of Mathematics and Statistics,
Mississippi State University,
175 President's Cir., Mississippi State, MS 39762, U.S.A.}
\email{mmcbride@math.msstate.edu}

\author[Peoples]{J. Wilson Peoples}
\address{Department of Mathematics,
Pennsylvania State University,
107 McAllister Bld., University Park, State College, PA 16802, U.S.A.}
\email{jwp5828@psu.edu}

\date{\today}

\begin{abstract}
We introduce and investigate some examples of C$^*$-algebras which are related to multiplication maps in the ring of $p$-adic integers. We find ideals within these algebras and use the corresponding short exact sequences to compute the $K$-Theory.  
\end{abstract}

\maketitle
\section{Introduction}
Many interesting examples of C$^*$-algebras with connections to number theory have been studied, see, for instance, \cite{cuntz2008c, li2014k}. Explorations of such algebras lead to intriguing problems and have inspired development of new C$^*$-algebraic techniques. 

One basic example of C$^*$-algebras related to number theory are the Bunce-Deddens (BD) algebras, a recent detailed study of which can be found in \cite{klimek2021aspects}. In special cases, BD algebras can be realized as crossed product algebras $C(\Z_p) \rtimes \Z$, where $C(\Z_p)$ denotes the continuous functions on the $p$-adic integers, and the automorphism implementing the action of $\Z$ makes use of the additive group structure of $\Z_p$. 

Since the $p$-adic integers also carry a ring structure, the above suggests studying maps $C(\Z_p) \to C(\Z_p)$ of the form 
$$f(x) \mapsto f(rx),$$ 
for $0\ne r \in \Z_p$, from a C$^*$-algebra perspective. Such maps are not always automorphisms, and hence the most natural setting for this investigation are the crossed products by endomorphism C$^*$-algebras. Early notions of crossed products by endomorphism were introduced by \cite{stacey1993crossed} and also \cite{paschke}. A similar, yet slightly different notion was suggested in \cite{exel2003new}, and further improved upon in \cite{exel2008new}. For the particular algebras considered in the present paper, these notions coincide. 

In this paper, we study the structure of C$^*$-algebras which are crossed products of $C(\Z_p)$ by endomorphisms corresponding to $p$-adic multiplication for three interesting cases of $p$-adic integers $r$ depending on the $p$-adic norm and the order of $r$. In particular, we relate those algebras with some familiar C$^*$-algebras and compute their $K$-Theory.

We start with a bit of necessary number theoretic preparations. We then discuss a closely related example of the Hensel-Steinitz C$^*$-algebra, a detailed study of which can be found in \cite{HKMP}. We conclude by providing an in depth exploration of crossed product of $C(\Z_p)$ by endomorphism associated with $p$-adic multiplication by $r$ for different special cases of $r \in \Z_p$. 

In particular, when $r \in \Z_p^\times$ is an element of the $p$-adic unit sphere $\Z_p^\times$, and depending on whether $r$ is or is not a root of unity in $\Z_p^\times$, we find geometrical representations of the crossed product which can be used to characterize an ideal in it, which can in turn be used to compute the $K$-Theory of the full algebra. On the other hand, when $r \in \Z_p$ is divisible by $p^N$ for some $N \geq 1$, we construct a C$^*$-algebra isomorphism with the Hensel-Steinitz C$^*$-algebra studied in \cite{HKMP}. 

The discussion separates into cases depending on the dynamics of the map $x\mapsto rx$, and the methods of proof are based on finding faithful representations of the crossed products. 

\section{C$^*$-Algebras Associated to p-Adic Multiplication }
\label{main algebra}
Here we introduce the main objects of  study in this paper. Given a $p$-adic integer $r \in \Z_p$, $r \neq 1$, consider the endomorphism $\alpha_r: C(\Z_p) \to C(\Z_p)$ given by 
\begin{equation*}
(\alpha_r f)(x) = \left\{
\begin{aligned}
&f\left(\frac{x}{r}\right) &&\textrm{ if }r|x \\
&0 &&\textrm{else.}
\end{aligned}\right.
\end{equation*}
Similarly, consider the map $\beta_r: C(\Z_p) \to C(\Z_p)$ given by 
\begin{equation*}
    (\beta_r f)(x) = f(rx). 
\end{equation*}
We have the following relation between $\alpha_r$ and $\beta_r$:
\begin{equation}
\label{alpha injective}
(\beta_r \circ \alpha_r)f(x) = f(x).
\end{equation}
In what follows, we investigate the structure of the crossed product C$^*$-algebra by endomorphism $\alpha_r$:
\begin{equation*}
    A_r : = C(\Z_p) \rtimes_{\alpha_r} \N,
\end{equation*}
utilizing the crossed product definition in \cite{stacey1993crossed}. Recall that the crossed product by endomorphism in our case is defined as the universal unital C$^*$-algebra generated by $C(\Z_p)$, with generators denoted as $M_f$, for $f \in C(\Z_p)$, and an additional generator $V_r$ satisfying relations 
$$
V_r^*V_r = I, \quad V_r M_f V_r^* = M_{\alpha_r(f)}. 
$$
It is immediate from the appendix in \cite{HKMP} that $\alpha_r$ is a monomorphism and that the range of $\alpha_r$ is hereditary.

We briefly describe how the resulting crossed products differ in three distinct cases. First is the case when $r \in \Z_p^\times$ and $r$ is not a root of unity, second is when $r$ is a root of unity, while the third is the case when $r \in \Z_p$ satisfies $|r|_p = p^{-N}$, $N>0$. The remainder of the paper is dedicated to a more in depth investigation of each case. 
\subsection{Case I: $C(\Z_p) \rtimes_{\alpha_r} \N$ when $r \in \Z_p^\times$ and $r$ is not a root of unity} A basic fact of $p$-adic integers is that if $|r|_p = 1$, then $r$ is invertible in $\Z_p$ with inverse satisfying  $|r^{-1}|_p = 1$. It follows that $\alpha_r$ is an automorphism. Indeed, we have that 
\begin{equation*}
    (\alpha_r \circ \beta_r f)(x) =  (\beta_r f)(r^{-1}x) = f(x),
\end{equation*}
in addition to equation \ref{alpha injective}. When this is the case, the crossed product by endomorphism coincides with the standard notion of a crossed product by automorphism. This is summarized in the following proposition. 
\begin{prop}
When $|r|_p = 1$, $\alpha_r$ is an automorphism and $A_r$ becomes the standard crossed product: 
\begin{equation*}
    A_r \cong C(\Z_p) \rtimes_{\alpha_r} \Z. 
\end{equation*}
\end{prop}
\begin{proof}
Since $I = M_1 = M_{\alpha_r(1)}$, where $1$ denotes the constant function in $C(\Z_p)$, we have that 
$$V_r V_r^* = I$$ and therefore $V_r$ is unitary.
Moreover, note that 
$$
M_f = V_r^*V_r M_f V_r^* V_r = V_r^* M_{\alpha_r(f)} V_r. 
$$
This is precisely $C(\Z_p) \rtimes_{\alpha_r} \Z$. This completes the proof. 
\end{proof}
We study the structure of this algebra thoroughly in Section \ref{Case I}. 

\subsection{Case II: $C(\Z_p) \rtimes_{\alpha_r} \N$ when $r \in \Z_p^\times$ and $r$ is a root of unity} 
The considerations above depend only on the invertibility of $r$ and not on  $r$ being a root of unity. 
Thus, we also have in this case:
\begin{equation*}
    A_r \cong C(\Z_p) \rtimes_{\alpha_r} \Z. 
\end{equation*}
However, if $r$ is a root of unity, a similar object of interest is the crossed product:
$$
A'_r := C(\Z_p) \rtimes_{\alpha_r} \Z/\!\ord(r)\Z,
$$
where $\ord(r)$ is the order of $r$ in $\Z_p^\times$ and is also the order of $\alpha_r$, the smallest positive integer $n$ such that $r^n=1$.
The algebras $A'_r$ and $A_r$ are closely related but not in an obvious way.  
The structure and $K$-Theory of those crossed product algebras is studied in Section \ref{Case II}.

\subsection{Case III: $C(\Z_p) \rtimes_{\alpha_r} \N$ when $r \in \Z_p$ and $|r|_p < 1$} When $|r|_p < 1$, $r$ is no longer invertible, and the notion of crossed product by endomorphism is needed. Following the definition in \cite{stacey1993crossed}, we have that $A_r$ is the universal C$^*$-algebra generated by elements of $C(\Z_p)$, along with an isometry $V_r$ satisfying the relation 
$V_r M_f V_r^* = M_{\alpha_r(f)}$.  Again, by the appendix in \cite{HKMP}  $\alpha_r$ is a monomorphism and  the range of $\alpha_r$ is hereditary.  
We give a complete study of the structure and $K$-Theory of this algebra in Section \ref{Case III}.

\section{Number Theory Preliminaries}
In this section, we review some group theory which is used in the construction and investigation of the algebras introduced in Section \ref{main algebra}. 
Specifically, given $r \in \Z_p^\times$ we consider the (closed) subgroup
$G_r : = \overline{\{ r^k : k \in \Z \}}$ of $\Z_p^\times$ and the quotient group $\Z_p^\times / G_r$. 
To describe $G_r$ and the quotient $\Z_p^\times / G_r$ one could use the well-known structure results on $\Z_p^\times$ that are obtained using p-adic logarithmic and exponential functions. Alternatively, and this is the more algebraic approach described below, one can use projective limit techniques.

First, we review the notion of a primitive root, and recall some of their useful properties and related concepts.

\subsection{Primitive Roots} 
Let $p$ be any odd prime, and denote by $U_N$ the set of elements of the ring $\Z / p^N \Z$ which have a multiplicative inverse. It is easy to see that 
$$
U_N = \{ 0 < k < p^N : \textup{gcd}(p,k) = 1 \},
$$
and that $U_N$ becomes a group when equipped with multiplication modulo $p^N$. The following theorem dates back to Gauss. 
\begin{theo}
(Gauss) $U_N$ is isomorphic to the cyclic group of order $p^{N-1}(p-1)$. 
\end{theo}
The above result gives rise to the following definition. 

\begin{defin}
We say that an element $a \in U_N \subseteq \Z / p^N \Z$ is a primitive root of $U_N$ if $a$ generates $U_N$. 
\end{defin}

Primitive roots enjoy the following useful properties:
\begin{itemize}
    \item If $a$ is a primitive root of $U_1$, then either $a$ or $a+p$ is a primitive root of $U_2$. 
    \item If $a$ is a primitive root of $U_N$ for $N \geq 2$, then $a$ is also a primitive root of $U_M$ for $M \geq N$.
\end{itemize}
\subsection{Subgroups of $U_N$} 
Next we consider subgroups of $U_N$ generated by a single  element $r$. In particular, the order of an element $r$ in $U_N$, denoted by $\ord_{U_N}(r)$, can be compared with its order in $U_{N+1}$ provided that $r$ satisfies certain properties. This result is summarized in the following proposition.
\begin{prop}
\label{order proposition}
Let $r$ be an integer not divisible by $p$ be such that there exists $M\geq 1$ such that $\ord_{U_M}(r)$ is divisible by $p$. Let $N_r$ denote the smallest such $M$. Then for all $N \geq N_r$, we have  
$$
\ord_{U_{N+1}} (r) = p \cdot \ord_{U_{N}}(r).  
$$
\end{prop}
\begin{proof}
Let $d_N$ denote the order of $r$ in $U_N$, and let $a$ be a primitive root of $U_N$ for all $N \geq 2$. Note that 
$$
r \equiv a^{\frac{p^{N}(p-1)}{d_{N + 1}}} \textup{ mod }  p^{N + 1},
$$
and hence
$$
r \equiv a^{\frac{p^{N}(p-1)}{d_{N + 1}}} \textup{ mod }  p^{N},
$$
since $p^{N} | p^{N + 1}$. Using a standard order formula for cyclic groups, we have that 
$$
d_{N} = \frac{p^{N -1 }(p-1)}{\textup{gcd}\left(p^{N - 1}(p-1), \frac{p^{N}(p-1)}{d_{N + 1}}\right)}. 
$$
Since $p | d_{N}$, and hence $p | d_{N + 1}$, we see that 
$$\textup{gcd}\left(p^{N - 1}(p-1), \frac{p^{N}(p-1)}{d_{N + 1}}\right) = \frac{p^{N}(p-1)}{d_{N + 1}}.$$
Hence, we obtain:
$$
d_{N} = \frac{p^{N -1 }(p-1)}{\frac{p^{N}(p-1)}{d_{N + 1}}} =  \frac{d_{N + 1}}{p}. 
$$
This completes the proof. 
\end{proof}
Similarly, letting $G_{r,N}$ denote the subgroup of $U_N$ generated by $r$,  the elements of $G_{r,M}$ can be described in terms of the elements of $G_{r,N}$ for $M \geq N$, provided that both $M$ and $N$ are sufficiently large. This description is summarized in the following proposition. 
\begin{prop}
\label{group proposition}
Let $r$ and $N_r$ be as in Proposition \ref{order proposition}. For all $N \geq N_r$, we have that 
$$
G_{r,N+1} = \{ k \in U_{N+1} : k \textup{ mod } p^N \in G_{r,N}\}. 
$$
\end{prop}
\begin{proof}
If $k \in G_{r,N+1}$, then there is $\ell$ such that 
$$r^\ell \equiv k \textup{ mod } p^{N+1}.$$ 
Clearly $r^\ell \equiv k \textup{ mod } p^N$, and so 
$$k \textup{ mod } p^N \in  G_{r,N}.$$ 
This shows 
$$G_{r,N+1} \subseteq \{ k \in U_{N+1} : k \textup{ mod } p^N \in G_{r,N}\}.$$ 
However, by the proof of Proposition \ref{order proposition}, the cardinality of $G_{r, N+1}$ is given by $pd_N$, where again $d_N$ denotes the order of $r$ in $U_N$. Moreover, elements of $$\{ k \in U_{N+1} : k \textup{ mod } p^N \in G_{r,N}\}$$ are given by 
$$
k = k' + p^N k'',
$$
where $k' \in G_{r,N}$, and $k'' \in \{ 0, \dots , p-1\}$. There are precisely $p d_N$ such elements. Hence, the sets coincide. This completes the proof. 
\end{proof}

\subsection{P-Adic Integers}
Given an integer $s>1$, and an increasing sequence of divisors $s | s^2 | s^3 \dots $, we have homomorphisms 
$$\pi_{n,m}: \Z / s^n \Z \to \Z / s^m \Z$$  
given by reduction mod $s^m$ for $n \geq m$. This gives rise to a directed system, whose inverse limit we define to be $\Z_s$: 
$$
\Z_s : = \varprojlim \Z / s^n \Z. 
$$
Of particular interest in this paper is the case when $s$ is an odd prime $p$ and $\Z_p$ is then the ring of p-adic integers. 

As $U_N \subseteq \Z / p^N\Z$ is the group of units of $\Z / p^N\Z$, we obtain for the group of units of $\Z_p$:
$$
\Z_p^\times : = \varprojlim U_N. 
$$
This is also termed the p-adic unit sphere, since it consists precisely of those $p$-adic integers with $p$-adic norm $1$, i.e. formal sums of the form $\sum_{i=0}^\infty a_i p^i$ which are not divisible by $p$. 

\subsection{Roots of Unity}
A corollary of Hensel's Lemma and Fermat's Little Theorem is that within $\Z_p^\times$, there exist $p-1$ roots of unity of orders dividing $p-1$. In other words, one can find p-adic integers $\omega_1, \omega_2 \dots , \omega_{p-1}$, such that for each $i$, 
$$
\omega_i^{p-1}  = 1.
$$
Moreover, simple divisibility arguments show that there are no other p-adic roots of unity. The group of all p-adic roots of unity will be denoted $\mathcal{G}_p$.

The above roots are typically arranged such that for every $0\leq i\leq p-1$ we have:
$$
\omega_i \equiv i \textup{ mod } p.
$$
It follows that any $z\in\Z_p^\times$ can be written uniquely as the product
\begin{equation}\label{root_decomp}
z=\omega_i y
\end{equation}
for some $0\leq i\leq p-1$ and $y\in 1+p\Z_p$. One can then write $y$ as a p-adic exponential, giving a nice description of the group structure of $\Z_p^\times$, though we will not use it in this paper.

\subsection{Groups $G_r$} 
We now investigate only those $r \in \Z_p^\times$ which are not roots of unity. Restricting to such $r$ allows for the following result which is a consequence of Proposition \ref{order proposition}. 
\begin{prop}
\label{order corollary}
Let $r \in \Z_p^\times$, and assume $r$ is not a root of unity. Then there exists $N_r$ such that for all $N \geq N_r$, 
$$
\ord_{U_{N+1}}(r) = p \cdot \ord_{U_N}(r). 
$$ 
\end{prop}
\begin{proof}
Let $d_N$ denote the order of $r$ in $U_N$. For any $N$, 
$$d_N | p^N (p-1).$$ 
However, since 
$$r^{d_{N+1}} \equiv 1 \textup{ mod } p^{N+1} \implies r^{d_{N+1}} \equiv 1 \textup{ mod } p^{N},$$
we see that $d_{N+1} \geq d_N$ and that the orders are increasing. Since $r$ is not a root of unity, and the orders grow as $N$ increases, we see that eventually $d_N$ must gain a power of $p$, i.e. there exists $N_r$ such that $d_{N_r}$ is divisible by $p$. Hence, the hypothesis of Proposition \ref{order proposition} is satisfied. The exact same argument contained in the proof of Proposition \ref{order proposition}  shows the final result. 
\end{proof}
The following proposition follows immediately from  Proposition \ref{group proposition}.
\begin{prop}
\label{group corollary}
Let $r \in \Z_p^\times$ and assume $r$ is not a root of unity. Then for all $N \geq N_r$, we have 
$$
G_{r, N+1} = \{k \in U_{N+1} : k \textup{ mod }p^N \in G_{r,N}\}. 
$$
\end{prop}

Given $r \in \Z_p^\times$ we consider the closed subgroup of $\Z_p^\times$ generated by $r$:
$$
G_r : = \overline{\{ r^k : k \in \Z \}}. 
$$
We have the following result regarding the quotient group $\Z_p^\times / G_r$. This is the main result of this section.
\begin{theo}
\label{discrete quotient}
Let $N_r$ be as in Proposition \ref{order corollary}, and assume $r \in \Z_p^\times$ is not a root of unity. We have the following isomorphism of groups:
$$
\Z_p^\times / G_r \cong U_{N_r} / G_{r, N_r}.     
$$
In particular, $\Z_p^\times/G_r$ is a finite group.
\end{theo}
\begin{proof}
We first prove that 
$$
\varprojlim G_{r,N} = G_r. 
$$
Note that 
$$
\varprojlim G_{r,N} : = \{ x \in \Z_p^\times : x \textup{ mod } p^N \in G_{r,N} \textup{ } \forall N\geq 1 \}. 
$$
Notice that $G_r \subseteq \varprojlim G_{r,N}$. Indeed, since powers of $r$ are  in $\varprojlim G_{r,N}$ and $\varprojlim G_{r,N}$ is closed,  we have that  inclusion for the closure:
$$\overline{\{r^k\}}_{k \in \Z} = G_r \subseteq \varprojlim G_{r,N}.$$ 
To see the other inclusion, let $x \in \varprojlim G_{r,N}$. For any $M \geq 1$ we have: 
$$x \textup{ mod } p^M \in G_{r,M}.$$ 
Hence, by definition there is $k_M$ such that 
$$r^{k_M} \equiv x \textup{ mod }p^M.$$ 
But this means that 
$$|r^{k_M} - x|_p \leq p^{-M}.$$  
Hence, $r^{k_M} \to x$ in $p$-adic norm as $M \to \infty$. Since $\{ r^{k_M}\}_{M=1}^\infty \subseteq G_r$, it follows $x \in G_r$. This proves that 
$$
\varprojlim G_{r,N} = G_r. 
$$

Next, we  give a more useful interpretation of $\varprojlim G_{r,N}$. If $x \in G_{r,N_r}$, then certainly $x \in G_{r, N}$ for $N \leq N_r$. However, by Proposition \ref{group corollary}, $G_{r, N_r + 1}$ is identified with those $k \in U_{N_r + 1}$ satisfying 
$$k \textup{ mod }p^{N_r} \in G_{r,N_r}.$$ 
Applying this proposition $M-N_r$ consecutive times, we see that $G_{r, M}$ consists of elements of the form  
$$
k_0 + p^{N_r} k_{N_r} + p^{N_r + 1} k_{N_r + 1} + \dots + p^{M-1}k_{M-1},
$$
where $k_0 \in G_{r,N_r}$ and $k_i \in \{0, \dots , p-1 \}$ for $i \neq 0$. This shows that
\begin{equation*}
\varprojlim G_{r,N} = \{ x \in \Z_p^\times : x \textup{ mod }p^{N_r} \in G_{r,N_r} \}. 
\end{equation*}
Consequently, the map 
$$\Z_p^\times / G_r \to U_{N_r} / G_{r,N_r}$$ 
given by 
$$
x G_r \mapsto (x \textup{ mod } p^{N_r})\, G_{r,N_r}
$$
is an isomorphism. 
Hence, we have that 
$$
\Z^\times_p / G_r =  \varprojlim U_N / \varprojlim G_{r,N} \cong U_{N_r} / G_{r, N_r}.
$$
This completes the proof. 
\end{proof}

\section{The Hensel-Steinitz Algebras}
\label{Hensel}
The Hensel-Steinitz algebras, introduced and studied in \cite{HKMP}, are closely related to the objects of study in this paper. In fact, we show in Section \ref{Case II} that for appropriate $r$, the C$^*$-algebra associated with $p$-adic multiplication by $r$ is an example of a Hensel-Steinitz algebra. In what follows, we discuss the definition of the Hensel-Steinitz algebras, and briefly review their structure and $K$-theory. For a more detailed discussion of the content of this section, we direct the reader to \cite{HKMP}. 
\subsection{Endomorphisms of $C(\Z_s)$} Let  $s>1$ be an integer, not necessarily a prime, and let $\Z_s$ be the $s$-adic ring defined in the previous section. Define a map $\alpha_s:C(\Z_s)\to C(\Z_s)$ by 
\begin{equation*}
(\alpha_s f)(x) = \left\{
\begin{aligned}
&f\left(\frac{x}{s}\right) &&\textrm{ if }s|x \\
&0 &&\textrm{else}
\end{aligned}\right.
\end{equation*}
Similarly, define a map $\beta_s:C(\Z_s)\to C(\Z_s)$ via 
$$(\beta_s f)(x) = f(sx).$$
We see that both maps are indeed endomorphisms of $C(\Z_s)$. We have the following relation between $\alpha_s$ and $\beta_s$:
\begin{equation*}
(\beta_s\circ\alpha_s)f(x) = \beta_s(\alpha_s(f))(x) = (\alpha_s f)(sx) = f(x)
\end{equation*}
for any $f\in C(\Z_s)$. Thus, $\alpha_s$ is an injection. 
\subsection{The Hensel-Steinitz Algebra as a Crossed Product}
We define the Hensel-Steinitz algebra to be the following crossed product by endomorphism $\alpha_s$:
\begin{equation*}
HS(s) = C(\Z_s) \rtimes_{\alpha_s}\N\,,
\end{equation*}
where the above crossed product by endomorphism is in the sense of \cite{stacey1993crossed}:  the universal unital C$^*$-algebra generated by the shift $V_s$ and multiplication operator $M_f$ satisfying the relations
$$V_s^* V_s = I, \ \ V_sM_fV_s^* = M_{\alpha_s(f)}.$$
This algebra can be represented faithfully on $\ell^2(\Z_{\geq 0})$ as the C$^*$-algebra generated by a shift, denoted below by $V$, and by multiplication operators. Indeed, consider $\ell^2(\Z_{\geq 0})$ and let $\{E_l\}_{l\in\Z_{\geq 0}}$ be its canonical basis. Since $\Z_{\geq 0}$ is a dense subset of $\Z_s$, the mapping 
$$C(\Z_s)\to B(\ell^2(\Z_{\geq 0})),\  f\mapsto \mu _f,$$ 
with $\mu_f$ given by 
$$\mu_fE_l = f(l)E_l,$$ 
is a faithful representation of $C(\Z_s)$ on $\ell^2(\Z_{\geq 0})$. Let $V$ be the following shift operator defined on $E_l$ by
\begin{equation}
\label{V shift}
VE_l = E_{sl}\,.
\end{equation}
A simple calculation verifies that
\begin{equation*}
V^*E_l = \left\{
\begin{aligned}
&E_{l/s} &&\textrm{ if } s|l \\
&0 &&\textrm{else.}
\end{aligned}\right.
\end{equation*}
We state here the following proposition, which is proved in \cite{HKMP}. 
\begin{prop}
We have the following isomorphism between C$^*$-algebras:
\begin{equation*}
HS(s) = C(\Z_s)\rtimes_\alpha\N \cong C^*(V,\mu_f: f \in C(\Z_s))\,.
\end{equation*}
\end{prop}
\subsection{The Structure and $K$-Theory of $HS(s)$} 
To understand the structure of $HS(s)$ we look at the ideal, denoted $J_s$, which is introduced in the following way.
Consider the map $C(\Z_s) \to B(\ell^2(\Z))$ which sends $f \mapsto m_f$, where $m_f$ is given by  
$$m_f E_l = f(0) E_l.$$
Additionally, let $v \in B(\ell^2(\Z))$ denote the standard bilateral shift 
$$vE_l = E_{l+1}.$$ 
It is easy to check that sending the universal generators $V_s \mapsto v$ and $M_f \mapsto m_f$ gives rise to a representation of $HS(s)$ onto $C^*(v,m_f) \cong C(S^1)$. Denote this surjective representation by 
$$\pi_0: HS(s) \to C(S^1).$$ 
We define $J_s$ to be the ideal which is the kernel of this representation:
$$
J_s : = \textup{Ker } \pi_0. 
$$
We now have the following proposition regarding the structure of $J_s$. For a detailed proof of the following result, see \cite{HKMP}. 
\begin{prop}
We have the following isomorphism of C$^*$-algebras:
\begin{equation*}
J_s\cong C(\Z_s^\times)\otimes\mathcal{K}\,,
\end{equation*}
where $\Z_s^\times$ denotes the unit sphere in $\Z_s$ with respect to the $s$-adic norm. 
\end{prop}
This fact leads to the following short exact sequence for $HS(s):$
$$
0 \to C(\Z_s^\times)\otimes\mathcal{K} \to HS(s) \to C(S^1) \to 0. 
$$
The above sequence can be used to compute the $K$-Theory of $HS(s)$. We state the results here for convenience, for further details see \cite{HKMP}. 
\begin{prop}
The $K$-Theory of $HS(s)$ is given by the following: 
$$
K_0(HS(s)) \cong  C(\Z^\times_s, \Z) \textup{ and } K_1(HS(s)) \cong 0.
$$
\end{prop}

The Hensel-Steinitz algebras play a key role in Section \ref{Case III}.

\section{Structure of the Crossed Product: Case I}
\label{Case I}
We study the structure of $A_r$ via an ideal, denoted by $I_r$, which gives rise to the following short exact sequence: 
$$
0 \to I_r \to A_r \to C(S^1) \to 0.
$$
We begin this section by introducing the ideal and corresponding short exact sequence. Then, after demonstrating a number of irreducible, infinite dimensional representations of $A_r$, we obtain a faithful representation of $A_r$ on the Hilbert space $\ell^2(\mathbb{Z} \times \mathbb{Z}_{\geq 0} \times U_{N_r} / G_{r,N_r})$, where $N_r$ is as in Corollary \ref{order corollary}. Using this faithful representation, we describe the ideal $I_r$ in terms of familiar C$^*$-algebras. This description allows the $K$-Theory of $I_r$ to be computed via the K\"unneth formula, which in turn allows the $K$-Theory of $A_r$ to be computed via the $6$-term exact sequence in $K$-Theory. 
Throughout this section, $r$ is used to denote an element of $\Z_p^\times$ which is not a root of unity.
\subsection{Representations of $A_r$}
Let $\{ E_k\}_{k\in \Z}$ denote the canonical basis for $\ell^2(\Z)$, and consider the representation $\pi_0: A_r \to B(\ell^2(\Z))$ given by the following:
\begin{equation*}
\begin{aligned}
&V_r \mapsto v &&\textrm{ where }\quad v E_k = E_{k+1} \\
&M_f \mapsto m_f &&\textrm{ where }\quad m_f E_k = f(0) E_k.
\end{aligned}
\end{equation*}
Denote the kernel of this representation by $I_r : = \textup{Ker} \,\pi_0$. From the same reasoning as in Section \ref{Hensel}, it is clear that the image of $\pi_0$ can be identified with $C(S^1)$, and we have the following proposition. 
\begin{prop}
We have the following short exact sequence: 
\begin{equation*}
    0 \to I_r \to A_r \to C(S^1) \to 0.
\end{equation*}
\end{prop}
There are a number of important representations which are related to the orbits $\{r^kx : k \in \Z \}$. We summarize these representations and characterize their images in the following proposition. 
\begin{prop}
\label{sphere rep}
Let $r \in \Z_p^\times$ and assume $r$ is not a root of unity. Let $N_r$ be as defined in Propositions \ref{order proposition} and \ref{group proposition}. For any fixed nonzero $x \in \Z_p$, the map $\pi_x: A_r \to B(\ell^2(\Z)) $ defined by 
\begin{equation*}
\begin{aligned}
    &V_r \mapsto v  &&\textrm{ where } \quad v E_k = E_{k+1} \\
    &M_f \mapsto M_{r,x}(f) &&\textrm{ where } \quad M_{r,x}(f) E_k = f(r^kx)E_k 
\end{aligned}
\end{equation*}
defines a representation of $A_r$. Moreover, there exists an isomorphism of C$^*$-algebras: 
$$ \pi_x (A_r) \cong \textup{BD}( \ord_{\Z^\times_p} (r)),$$ 
where $\textup{BD}(S)$ denotes the Bunce-Deddens algebra with supernatural number $S$, and $\ord_{\Z^\times_p} (r)$ denotes the supernatural number given by $\textup{lcm} \{ \ord_{U_N}(r) : N \geq 1\}$. 
\end{prop}
\begin{proof}
It is clear that $\pi_x$ defines a representation of $A_r$. To see that the image can be identified with a Bunce-Deddens algebra, consider a locally constant (\textit{i.e.,} periodic) function $f$ on $\Z_p$ of period $p^N$. We examine the period of the sequence $k \mapsto f(r^k x)$. Note that by definition $|G_{r,N}|$ is the smallest number such that $r^{k+|G_{r,N}|} \equiv r^k (\textup{mod } p^N)$ for any $k$. Hence $r^{|G_{r,N}|} - 1$ is divisible by $p^N$, and therefore so is $xr^k(r^{|G_{r,N}|} - 1)$. It follows that for any $k$ we have $f(x r^k) = f(x r^{k + |G_{r,N}|})$ and that the sequence $k \mapsto f(r^k x)$ has period $|G_{r,N}|$. 
Since periodic functions of arbitrarily large period are dense in $C(\Z_p)$, by the Stone-Weierstrass Theorem, it follows that $\pi_x(A_{r})$ is generated by the unitary shift operator $v$, along with periodic diagonal operators of period $|G_{r,N}|$, for any $N \geq 1$. Hence it follows from the definition of Bunce-Deddens algebras that $\pi_x(A_{r}) \cong \textup{BD}( \ord_{\Z_p^\times}(r))$. 
\end{proof}
\begin{rem}
We note that since $|r|_p = 1$, the orbit $\{ x r^k : k \in \Z\}$ lies in the $p$-adic sphere with radius $|x|_p$, and hence the representation $\pi_x(f)$ depends only on the values of $f$ on that sphere. More specifically, $\pi_x(f)$ depends only on $f$ restricted to $\overline{\{ r^kx : k \in \Z \}}$. 
\end{rem}

\subsection{The Structure of $I_r$ and $A_r$}
Using the decomposition of $\Z_p$ into orbits of $G_r$, we have the following theorem describing a faithful representation of $A_r$.
\begin{theo}\label{ideal_structure}
      Suppose $r \in \Z_p^\times$ and assume that $r$ is not a root of unity. Let $\gamma : \Z^\times_p / G_r \to \Z^\times_p$ denote a section of the quotient map $\Z_p^\times \to \Z_p^\times / G_r$. The following map defines a faithful representation of $A_r:$ 
      $$
      \pi : A_r \to B(\ell^2(\Z \times \Z_{\geq 0} \times \Z_p^\times / G_r)),
      $$ 
      given by 
    \begin{equation*}
\begin{aligned}
        &V_r \mapsto v_r&&\textrm{ where } \quad v_r \phi(k,L,xG_r) = \phi(k-1,L,xG_r)  \\
        &M_f \mapsto \mathcal{M}_{f}&&\textrm{ where } \quad \mathcal{M}_f \phi(k,L,xG_r) = f(r^k p^L \cdot \gamma(xG_r)) \phi(k,L,xG_r). 
\end{aligned}   
\end{equation*}
\end{theo}
\begin{proof}
First note that 
$$
v_r^*\mathcal{M}_f v_r \phi(k,L,xG_r) = f(r^{k+1}p^L \gamma(xG_r)) \phi(k,L,xG_r) = \mathcal{M}_{\beta_r(f)} \phi(k,L,xG_r),
$$
so the relations are preserved. Before proceeding, we remark that since the quotient $\Z^\times_p / G_r$ is discrete, by Theorem \ref{discrete quotient}, the map $\gamma : \Z^\times_p / G_r \to \Z_p^\times$ is automatically continuous. Additionally, the representation is faithful on the subalgebra $C^*(\mathcal{M}_f : f \in C(\Z_p))$. Indeed, suppose that $f(r^k p^L \gamma(xG_r)) = 0$ for all $L,k, x$. We show that $f$ must vanish on each sphere. Fix $\varepsilon>0$. We wish to show that for any $x \in \Z_p^\times$, $|f(xp^L)| \leq \varepsilon$. Since $f$ is continuous, choose $\delta$ so that $|f(y) - f(p^Lx)|_p < \varepsilon$ whenever $|x-y|_p < \delta$.  Note that for any $x \in \Z_p^\times$, there exists $g$ in $G_r$ such that $x = \gamma(xG_r) \cdot g$. Since $g \in G_r$, there is $k$ such that $|g - r^k|_p < \delta$. Note that 
$$
|p^L x - p^L r^k \gamma(xG_r)| = |p^L \gamma(xG_r) \cdot g - p^L r^k \gamma(xG_r)|_p = p^{-L}|g - r^k|_p < \delta,
$$
and so 
$$
| f(p^L x) - f(p^L r^k \gamma(xG_r)|_p = | f(p^L x) |_p < \varepsilon. 
$$
Hence, $f$ vanishes on each sphere and so $f = 0$. This shows the representation is faithful on $C^*(\mathcal{M}_f : f \in C(\Z_p))$. 

To see that the representation is faithful, we use the adaptation of the O'Donovan condition to crossed products by endomorphism, as described in \cite{boyd1993faithful}; see also the appendix to \cite{HKMP}. Define a one-parameter group of unitaries
$$U_r(\theta) : \ell^2(\Z \times \Z_{\geq 0} \times \Z_p^\times / G_r) \to \ell^2(\Z \times \Z_{\geq 0} \times \Z_p^\times / G_r)$$ by 
$$
U_r(\theta) \phi(k,L,xG_r) := e^{2 \pi i k \theta} \phi(k,L,xG_r). 
$$
Next, for $a \in C^*(v_r, \mathcal{M}_f)$, we define $\rho_\theta(a)$ by
$$
\rho_\theta (a) : = U_r(\theta) a U_r(-\theta). 
$$
Note that on generators of $A_r$ morphisms $\rho_\theta$ act as follows:
$$
\rho_\theta(v_r)  = e^{2 \pi i \theta} v_r, 
$$
and 
$$
\rho_\theta(\mathcal{M}_f) = \mathcal{M}_f,
$$
so $\rho_\theta$ is a continuous one-parameter group of automorphisms of $A_r$ with invariant subalgebra given by $C^*(\mathcal{M}_f : f \in C(\Z_p))$. 

Consider an expectation defined by
$$
\mathbb{E}(a) := \int_0^1 \rho_\theta(a) d\theta.
$$
Clearly we have: 
$$
\| \mathbb{E}(a) \| \leq \int_0^1 \| U_r(\theta) \| \|a \| \|U_r(-\theta) \| d \theta \leq \|a\|. 
$$
Given $a$ of the form $a = \sum_{n \in \Z} v_r^n \mathcal{M}_{f_n}$, we see that
$$E(a)=\mathcal{M}_{f_0},
$$
and so, combining the two observations we obtain
$$
\|\mathcal{M}_{f_0} \| \leq \|a \|. 
$$
Hence, the representation satisfies O'Donovan's conditions and is therefore faithful. This completes the proof. 
\end{proof}
In light of this representation, we can easily deduce the structure of $I_r$. 
\begin{prop}
\label{I_r structure}
Suppose that $r \in \Z_p^\times$ is not a root of unity. There exists an isomorphism of C$^*$-algebras:
$$
I_r \cong c_0(\Z_{\geq 0}) \otimes \textup{BD}( \ord_{\Z^\times_p} (r)).
$$ 
\end{prop}
\begin{proof}
Notice that since $\pi_0(I_r) = 0$ by definition,   $\mathcal{M}_f \in I_r$ if and only if  $f(0)=0$ for $f\in C(\Z_p)$. Consequently, $I_r$ can be described as follows:
\begin{equation*}
I_r=C^*(v_r^n\mathcal{M}_f : n\in \Z, f \in C(\Z_p), f(0)=0).
\end{equation*}

Next, we decompose $\Z_p$ into orbits of $G_r$. Any nonzero $x\in\Z_p$ can be uniquely written as a product:
$$x=p^L \, \gamma(jG_r)\,g,$$ 
where $L\in \Z_{\geq 0}$, $g\in G_r$, and $jG_r\in \Z_p^\times/G_r$.

If $f(0)=0$ and $f\in C(\Z_p)$, then by continuity we get:
\begin{equation*}
\lim_{L\to\infty}f(p^L \, \gamma(jG_r)\,g)=0.
\end{equation*}
This limit is uniform in $j$ and $g$ because $\Z_p^\times/G_r$ is finite and $G_r$ is compact. Finiteness of $\Z_p^\times/G_r$ also implies  that $G_r$ is open. It follows from the Stone-Weierstrass Theorem that we can identify the space of continuous functions on $\Z_p$ vanishing at zero with the space of sequences indexed by $\Z_{\geq 0} \times \Z_p^\times/G_r$ and converging to $0$ with values in $C(G_r)$:
\begin{equation*}
c_0(\Z_{\geq 0} \times \Z_p^\times/G_r, C(G_r)).
\end{equation*}
Notice that in the above identification, the action of $\alpha_r$ is only on $C(G_r)$. 

Denote  by $\widehat\pi$ the representation of $C(G_r) \rtimes_{\alpha_r} \Z$ in $\ell^2(\Z)$ generated $\widehat M(f)$, $f\in C(G_r)$ and by $v$, defined by 
\begin{equation*}
\begin{aligned}
vE_k &= E_{k+1}  \\
\widehat M(f) E_k &= f(r^k)E_k.  
\end{aligned}
\end{equation*}
Notice that $\widehat\pi$ has image $*$-isomorphic with the Bunce-Deddens algebra $\textup{BD}(\ord_{\Z^\times_p} (r))$ by Proposition \ref{sphere rep} and thus $C(G_r) \rtimes_{\alpha_r} \Z\cong \textup{BD}(\ord_{\Z^\times_p} (r))$.

It is  convenient to identify the Hilbert space of the representation $\pi$ as follows:
\begin{equation*}
\ell^2(\Z \times \Z_{\geq 0} \times \Z_p^\times / G_r)\cong \ell^2(\Z_{\geq 0} \times \Z_p^\times / G_r, \ell^2(\Z)).
\end{equation*}
With this identification the operator $v_r$ becomes the bilateral shift $v$ acting on values of sequences in $\ell^2(\Z_{\geq 0} \times \Z_p^\times / G_r, \ell^2(\Z))$.

Now $I_r$ can be identified with the algebra of sequences indexed by $\Z_{\geq 0} \times \Z_p^\times/G_r$ and converging to $0$ with values in $\textup{BD}(\ord_{\Z^\times_p} (r))$ :
\begin{equation}
\label{c0bd}
c_0(\Z_{\geq 0} \times \Z_p^\times/G_r, \textup{BD}(\ord_{\Z^\times_p} (r))).
\end{equation}
Indeed, this algebra acts faithfully in $\ell^2(\Z_{\geq 0} \times \Z_p^\times / G_r, \ell^2(\Z))$ by acting pointwise on values of sequences in $\ell^2(\Z_{\geq 0} \times \Z_p^\times / G_r, \ell^2(\Z))$:
\begin{equation*}
\phi(L,jG_r)\mapsto \widehat\pi(F(L,jG_r))\phi(L,jG_r),
\end{equation*}
where, for every $L$ and $jG_r$, we have $\phi(L,jG_r)\in \ell^2(\Z)$ and $F(L,jG_r)\in \textup{BD}(\ord_{\Z^\times_p} (r))$.
Clearly, the algebra generated by those operators, by above identifications, coincides with the algebra generated by $v_r^n\mathcal{M}_f$ where $n\in \Z$ and $f \in C(\Z_p)$ with $f(0)=0$.

The above algebra, equation (\ref{c0bd}), can  be identified with $c_0(\Z_{\geq 0} \times \Z_p^\times/G_r) \otimes \textup{BD}(\ord_{\Z^\times_p} (r))$. Since $\Z_{\geq 0} \times \Z_p^\times/G_r$ is countable, it is  bijective with $\Z_{\geq 0}$ and we get 
$$c_0(\Z_{\geq 0} \times \Z_p^\times/G_r)\cong c_0(\Z_{\geq 0}).$$
This completes the proof. 
\end{proof}
\subsection{\textit{K}-Theory}
The above results allow the $K$-Theory of both $I_r$ and $A_r$ to be computed with existing tools. First, we compute the $K$-Theory of $I_r$ using Proposition \ref{I_r structure} along with the K\"unneth formula. We need the following additive subgroup of $\Q$:
\begin{equation}
\label{Hs}
H_S : =\left \{ \frac{k}{l} \in \Q : k\in\Z, l | S\right\}
\end{equation}
for a supernatural number $S$.

\begin{prop}\label{kunn_prop}
Suppose $r \in \Z_p^\times$ is not a root of unity. The $K$-Theory of $I_r$ is given by 
$$
K_0(I_r) \cong c_0\left(\Z_{\geq 0}, H_{\ord_{\Z^\times_p} (r)}\right) \textrm{ and } K_1(I_r) \cong c_0(\Z_{\geq 0}, \Z)
$$
where $H_{\ord_{\Z^\times_p} (r)}$ is the $K_0$-group of the Bunce-Deddens algebra with supernatural number $\ord_{\Z^\times_p} (r)$, as defined by equation (\ref{Hs}) as well as \cite{klimek2021aspects}. Here  the groups $c_0(\Z_{\geq 0}, \Z)$ as well as $c_0\left(\Z_{\geq 0}, H_{\ord_{\Z^\times_p} (r)}\right)$ denote sequences that are eventually zero with values in $\Z$ and $H_{\ord_{\Z^\times_p} (r)}$ respectively.
\end{prop}
\begin{proof}
Note first that $c_0(\Z_{\geq 0})$ is an inductive limit of finite dimensional algebras of sequences which are eventually $0$. Hence,  $c_0(\Z_{\geq 0})$ is AF and $K_1(c_0(\Z_{\geq 0})) \cong 0$ (by Exercise 8.7 in \cite{rordam2000introduction}), while 
$$
K_0(c_0(\Z_{\geq 0})) \cong c_0(\Z_{\geq 0} , \Z),
$$
which follows from Exercise 3.4 in \cite{rordam2000introduction}. The $K$-Theory of $\textup{BD}(\ord_{\Z^\times_p} (r))$ was computed in \cite{klimek2021aspects}, and is given by
$$
K_0 (\textup{BD}(\ord_{\Z^\times_p} (r))) \cong H_{\ord_{\Z^\times_p} (r)}, \textup{ and } K_1(\textup{BD}(\ord_{\Z^\times_p} (r))) = \Z.
$$

Since $\textup{BD}(\ord_{\Z^\times_p} (r))$ is in the Bootstrap Class, and the $K$-Theory of $c_0(\Z_{\geq 0})$ is torsion free, by Proposition 2.14 in \cite{schochet1982topological}, we have isomorphisms 
$$
K_0( c_0(\Z_{\geq 0}) \otimes \textup{BD}(\ord_{\Z^\times_p} (r)) ) \cong (c_0(\Z_{\geq 0}, \Z) \otimes H_{\ord_{\Z^\times_p} (r)}) \oplus ( \Z \otimes 0) \cong  c_0(\Z_{\geq 0}, \Z) \otimes H_{\ord_{\Z^\times_p} (r)},
$$
as well as
$$
K_1(c_0(\Z_{\geq 0}) \otimes \textup{BD}(\ord_{\Z^\times_p} (r)) ) \cong  (0 \otimes \Z) \oplus (c_0(\Z_{\geq 0}, \Z) \otimes \Z ) \cong c_0(\Z_{\geq 0}, \Z).
$$
Since we have the identification
\begin{equation*}
c_0(\Z_{\geq 0}, \Z) \otimes H_{\ord_{\Z^\times_p} (r)} \cong c_0\left(\Z_{\geq 0}, H_{\ord_{\Z^\times_p} (r)}\right),
\end{equation*}
Proposition \ref{I_r structure}  completes the proof. 
\end{proof}
Finally, the $K$-Theory of $A_r$ can be computed from the $6$-term exact sequence induced by the short exact sequence $$0 \to I_r \to A_r \to C(S^1) \to 0.$$ 
\begin{prop}\label{ktheory_prop}
Suppose $r \in \Z_p^\times$ is not a root of unity. The $K$-Theory of $A_r$ is given by 
$$
K_0(A_r) \cong c_0(\Z_{\geq 0}, H_{\ord_{\Z^\times_p} (r)}) \oplus \Z  \textup{ and } K_1(A_r) \cong \Z \oplus c_0(\Z_{\geq 0}, \Z).
$$
\end{prop}
\begin{proof} 
The $6$-term exact sequence in $K$-Theory is given by 
\begin{equation*}
\begin{tikzcd}
 K_0(I_r) \arrow{r}   & K_0(A_r)  \arrow{r} & K_0(C(S^1)) \arrow{d}{\textup{ exp }}    \\
 K_1(C(S^1)) \arrow{u}{\textup{ ind }} & K_1(A_r) \arrow{l} & K_1(I_r) \arrow{l}
\end{tikzcd}
\end{equation*}
First, since $v$, the unitary in $C(S^1)$ whose class $[v]_1$ generates $K_1(C(S^1))$, lifts to a unitary $V_r \in A_r$, we have that 
$$
\textup{ind}([v]_1) = [1 - V_r^*V_r]_0 - [1 - V_rV_r^*]_0 = 0,
$$
by Proposition 9.2.4 in \cite{rordam2000introduction}. Similarly, since $[I]_0$ generates $K_0(C(S^1))$, it follows from Proposition 12.2.2 in \cite{rordam2000introduction} that 
$$
\textup{exp}([I]_0) = 0. 
$$
Hence, we can extract a short exact sequence of groups from the top row: 
$$
0 \to K_0(I_r) \to K_0(A_r) \to K_0(C(S^1)) \to 0. 
$$
It is clear this sequence is right split via the map $K_0(C(S^1)) \ni [I]_0 \mapsto [I]_0 \in K_0(A_r)$. Hence, by the Splitting Lemma 
$$
K_0(A_r) \cong  c_0(\Z_{\geq 0}, H_{\ord_{\Z^\times_p} (r)})\oplus \mathbb{Z}.
$$
Additionally, we can extract a short exact sequence of groups from the bottom row: 
$$
0 \to K_1(I_r) \to K_1(A_r) \to K_1(C(S^1)) \to 0. 
$$
This sequence is also right split via the map $K_1(C(S^1)) \ni [v]_1 \mapsto [V_r] \in K_1(A_r)$. Hence,
$$
K_1(A_r) \cong \Z \oplus c_0(\Z_{\geq 0}, \Z). 
$$
This completes the proof. 
\end{proof}
We remark that the above calculation shows that for any unital C$^*$-algebra $A$ with short exact sequence $$0 \to I \to A \to C(S^1) \to 0$$ such that the generating unitary $v \in C(S^1)$ lifts to a unitary in $A$, we have the following isomorphisms 
$$
K_0(A) \cong K_0(I) \oplus \Z, \textup{ and } K_1(A) \cong K_1(I) \oplus \Z.  
$$
This is a useful and simple result that is hard to find explicitly stated in the literature. 

\section{Structure of the Crossed Product: Case II}
\label{Case II}
We analyze here concurrently both algebras $$A'_r := C(\Z_p) \rtimes_{\alpha_r} \Z/\ord(r)\Z \quad \textrm{and} \quad A_r \cong C(\Z_p) \rtimes_{\alpha_r} \Z$$ when $r$ is of finite order $\ord(r)$. The analysis parallels analogous considerations of the previous section.

First we identify the ideals $I_r'$ and $I_r$ as kernels of the representations associated to the fixed point $0\in\Z_p$ of the map $x\mapsto px$. This was already done in the previous section for $A_r$. We proceed similarly for $A_r'$. 

Let $\{ E_k\}_{k\in \Z/\ord(r)\Z}$ be the canonical basis for $\ell^2(\Z/\ord(r)\Z)$, and consider a finite dimensional representation $\pi_0': A'_r \to B(\ell^2(\Z/\ord(r)\Z))$ given by the following:
\begin{equation*}
\begin{aligned}
&V_r \mapsto v'  &&\textup{where} \quad v' E_k = E_{k+1\text{ mod }(\ord(r))}  \\
&M_f \mapsto m_f'  &&\textup{where} \quad  m_f' E_k = f(0) E_k. 
\end{aligned}
\end{equation*}
Denote the kernel of this representation by $I'_r : = \textup{ker}\, \pi_0'$. 
We have the following short exact sequences: 
\begin{equation*}
    0 \to I_r' \to A_r' \to C(\Z/\ord(r)\Z) \to 0,
\end{equation*}
and, as before,
\begin{equation*}
    0 \to I_r \to A_r \to C(S^1) \to 0.
\end{equation*}

For any fixed nonzero $x \in \Z_p$, representations $\pi_x$ of $A_r$ related to the orbits $\{r^kx : k \in \Z \}$ were defined in Proposition \ref{sphere rep}.
Similarly, since the orbits are finite, we have the following finite dimensional representations 
 $\pi_x': A_r' \to B(\ell^2(\Z/\ord(r)\Z)) $ defined by 
\begin{equation*}
\begin{aligned}
    &V_r \mapsto V_{r}'  &&\textup{where } &V_{r}' E_k = E_{k+1\text{ mod }(\ord(r))}  \\
    &M_f \mapsto M_{r,x}'  &&\textup{where } &M_{r,x}'(f) E_k = f(r^kx)E_k  .
\end{aligned}
\end{equation*}

\begin{prop}
\label{sphere_rep_prime}
Suppose that $r \in \Z_p^\times$ is a p-adic root of unity. There exist isomorphisms of C$^*$-algebras: 
$$ \pi_x' (A_r') \cong M_{\ord(r)}(\C),$$ 
and 
$$ \pi_x (A_r) \cong C(S^1)\otimes M_{\ord(r)}(\C).$$
\end{prop}
\begin{proof} Since the representation $\pi_x'$ is irreducible, $\pi_x' (A_r')$ must be the full matrix algebra.

To see the claimed isomorphism for $\pi_x (A_r)$ 
, we observe that the algebra is generated by the following set of generators:
\begin{equation*}
u:=v^{\ord(r)},
\end{equation*}
and
\begin{equation*}
P_{i,j}:=v^i P_0v^{-j},
\end{equation*}
for $i,j=0,\ldots, \ord(r)-1$. Here  $P_0$ is the orthogonal projection onto the span of $E_k$ with $k$ divisible by $\ord(r)$.

Indeed,  it is easy to see that the subalgebra generated by all $M_{r,x}(f)$ is just the algebra of all diagonal operators in $\ell^2(\Z)$ whose diagonal elements are $\ord(r)$ periodic. But this is precisely the algebra generated by $P_{i,i}$, $i=0,\ldots, \ord(r)-1$.
Additionally a simple calculation shows that we have:
\begin{equation*}
v=P_{1,0}+P_{2,1}+\ldots+P_{\ord,(r)-1,\ord(r)-2}+uP_{0,\ord(r)-1}.
\end{equation*}
Thus, one can express $v$ and $M_{r,x}(f)$'s through $u$ and $P_{i,j}$'s and the other way around. It is easy to see that $P_{i,j}$ are units for $M_{\ord(r)}(\C)$ and $u$ commutes with  $P_{i,j}$ and thus together they generate $C(S^1)\otimes M_{\ord(r)}(\C)$.
\end{proof}

Let $\mathcal{G}_p$ be the group of all p-adic roots of unity and, as before, let $G_r$ be its subgroup generated by $r$. 
The following theorem describes faithful representations of $A_r$ and $A_r'$ in this case.
\begin{prop}
      Suppose that $r$ is a root of unity.  Let $\gamma : \mathcal{G}_p / G_r \to \mathcal{G}_p$ be a section of the quotient map $\mathcal{G}_p \to \mathcal{G}_p / G_r$.
      The following is a faithful representation of $A_r':$ 
      $$
      \pi' : A_r' \to B(\ell^2(\Z/\ord(r)\Z \times \Z_{\geq 0} \times \mathcal{G}_p / G_r \times \Z)),
      $$ 
      given by 
    \begin{equation*}
\begin{aligned}
        &V_r \mapsto v_r'  &\textup{where } &v_r' \phi(k,L,\omega G_r,x) = \phi(k-1\text{ mod }(\ord(r)),L,\omega G_r,x)   \\
        &M_f \mapsto M_{f}'  &\textup{where } &M_f' \phi(k,L,\omega G_r,x) = f(r^k p^L \cdot \gamma(\omega G_r)(1+px)) \phi(k,L,\omega G_r,x). 
\end{aligned}    \end{equation*}
  Similarly, we have a faithful representation of $A_r:$ 
      $$
      \pi : A_r \to B(\ell^2(\Z\times \Z_{\geq 0} \times \mathcal{G}_p / G_r \times \Z)),
      $$ 
      given by 
    \begin{equation*}
\begin{aligned}
        &V_r \mapsto v_r &\textup{where }& v_r \phi(k,L,\omega G_r,x) = \phi(k-1,L,\omega G_r,x)   \\
        &M_f \mapsto \mathcal{M}_{f}  &\textup{where }& \mathcal{M}_f \phi(k,L,\omega G_r,x) = f(r^k p^L \cdot \gamma(\omega G_r)(1+px)) \phi(k,L,\omega G_r,x). 
\end{aligned}    \end{equation*}
\end{prop}
\begin{proof} From the definition of $\gamma$, any root of unity $\omega\in \mathcal{G}_p$ can be written as $r^k \cdot \gamma(\omega G_r)$ for some $k$.
Then, utilizing equation \eqref{root_decomp}, we can write any $z\ne 0$, with $|z|_p=p^{-L}$, as:
$$
z=r^k p^L \cdot \gamma(\omega G_r)(1+px)
$$
for some $x\in\Z_p$. Since $\Z$ is dense in $\Z_p$, it follows that representations $\pi'$ and $\pi$ are faithful on $C(\Z_p)$.

The proof of the remaining part of O'Donovan's conditions for the representation $\pi$ just repeats the arguments from Theorem \ref{ideal_structure}.  For the representation $\pi'$ it requires the following modification. Instead of a one-parameter group of unitaries $U_r(\theta)$, consider diagonal unitary operators $W_r(j)$ given by:
\begin{equation*}
W_r(j) \phi(k,L,\omega G_r,x)=e^{\frac{2\pi ijk}{\ord(r)}} \phi(k,L,\omega G_r,x),
\end{equation*}
with $j=0,1,\ldots, \ord(r)-1$. Clearly, we have:
\begin{equation*}
W_r(j)W_r(j')=W_r(j+j' \ \text{ mod }(\ord(r))).
\end{equation*}
Next, for $a \in C^*(v_r', M_f')$, we define maps $\rho_j$ by
$$
\rho_j (a) : =W_r(j) a W_r(-j). 
$$
Then $\rho_j$ are automorphisms of $C^*(v_r', M_f')$ and we can use the following expectation:
\begin{equation*}
E(a):=\frac{1}{\ord(r)}\sum_{j=0}^{\ord(r)-1}\rho_j (a)
\end{equation*}
to verify the remaining O'Donovan's condition.
\end{proof}

The above faithful representations let us easily deduce the structure of $I_r'$ and $I_r$. 
\begin{prop}
Suppose that $r \in \Z_p^\times$ is a root of unity. There exist isomorphisms of C$^*$-algebras:
$$
I_r' \cong M_{\ord(r)}(\C) \otimes c_0(\Z_{\geq 0}) \otimes C(\Z_p)
$$ 
and 
$$
I_r \cong I_r' \otimes C(S^1).
$$
\end{prop}
\begin{proof} Just as in the proof of Proposition \ref{I_r structure}, the statements are consequences of the previous two propositions. We describe details for $I_r'$. From the definition we have
\begin{equation*}
I_r'=C^*((v_r')^nM'_f : n\in \Z, f \in C(\Z_p), f(0)=0).
\end{equation*}
Then for nonzero $z$, using the decomposition $z=r^k p^L \cdot \gamma(\omega G_r)(1+px)$, we get an identification of
the space of continuous functions on $\Z_p$ vanishing at zero with the space of continuous functions vanishing at infinity on a locally compact space $\Z_{\geq 0} \times \Z_p^\times/G_r$ with values in $C(G_r)$:
\begin{equation*}
C_0(\Z_{\geq 0}\times  \mathcal{G}_p / G_r  \times \Z_p, C(G_r)).
\end{equation*}

Next, it is convenient to view the Hilbert space of the representation $\pi'$ as:
\begin{equation*}
\ell^2(\Z/\ord(r)\Z \times \Z_{\geq 0} \times \mathcal{G}_p / G_r \times \Z)\cong 
\ell^2( \Z_{\geq 0} \times \mathcal{G}_p / G_r \times \Z, \ell^2(\Z/\ord(r)\Z)).
\end{equation*}
With this identification the operator $v_r'$ becomes the mod $\ord(r)$ shift acting on the space $\ell^2(\Z/\ord(r)\Z)$ of values of sequences in the Hilbert space of the representation $\pi'$. This gives an identification:
\begin{equation*}
I_r'\cong C_0(\Z_{\geq 0}\times  \mathcal{G}_p / G_r  \times \Z_p, C(G_r) \rtimes_{\alpha_r} \Z/\ord(r)\Z).
\end{equation*}
Again, up to isomorphism, one can drop from above the finite set $ \mathcal{G}_p / G_r$, which, together with Proposition \ref {sphere_rep_prime} proves the claimed structure of $I_r'$. The details for $I_r$ are completely analogous.
\end{proof}

The above results allow us again to compute the $K$-Theory of both $I_r'$ and $I_r$.
\begin{prop}
Suppose $r \in \Z_p^\times$ is a root of unity. The $K$-groups of $I_r'$ and $I_r$ are given by 
$$
K_0(I_r') \cong c_0(\Z_{\geq 0}\times \Z_p, \Z) , \textup{ and } K_1(I_r') = 0,
$$
and 
$$
K_0(I_r) \cong c_0(\Z_{\geq 0}\times \Z_p, \Z) , \textup{ and } K_1(I_r)  \cong c_0(\Z_{\geq 0}\times \Z_p, \Z).
$$
\end{prop}
\begin{proof} By stability of $K_i$ groups we can ignore the $ M_{\ord(r)}(\C)$ factor in $I_r'$. Then the calculation, using K\"unneth's formula, is the same as in the proof of Proposition \ref{kunn_prop}.
\end{proof}

Now it is routine again to compute the $K$-Theory of both $A_r'$ and $A_r$.
\begin{prop}
Suppose $r \in \Z_p^\times$ is a root of unity. We have: 
$$
K_0(A_r') \cong c_0(\Z_{\geq 0}\times \Z_p, \Z) \oplus \Z^{\ord(r)}, \textup{ and } K_1(A_r') = 0,
$$
and 
$$
K_0(A_r) \cong c_0(\Z_{\geq 0}\times \Z_p, \Z)\oplus \Z , \textup{ and } K_1(A_r)  \cong c_0(\Z_{\geq 0}\times \Z_p, \Z)\oplus \Z.
$$
\end{prop}
\begin{proof} Again, as in Proposition \ref{ktheory_prop}, both exponential and index maps in the 6-term exact sequence below are zero:
\begin{equation*}
\begin{tikzcd}
 K_0(I_r') \arrow{r}   & K_0(A_r')  \arrow{r} & K_0(C(\Z/\ord(r)\Z)) \arrow{d}{\textup{ exp }}    \\
 K_1(C(\Z/\ord(r)\Z)) \arrow{u}{\textup{ ind }} & K_1(A_r') \arrow{l} & K_1(I_r') \arrow{l}
\end{tikzcd}
\end{equation*}
Thus, all the $K_1$ groups in the bottom row are zero. The short exact sequence in the top row splits by the same argument as before, yielding the result for $K_0(A_r')$ since we have:
\begin{equation*}
K_0(C(\Z/\ord(r)\Z))\cong \Z^{\ord(r)}.
\end{equation*}

The formulas for $K_i(A_r)$ follow immediately from the previous proposition and the remark immediately following the proof of Proposition \ref{ktheory_prop}.
\end{proof}

\section{Structure of the Crossed Product: Case III} 
\label{Case III} 
In this section we find that when $|r|_p = p^{-N}$, with $N>0$, the crossed product can be identified with a Hensel-Steinitz algebra. This is done by constructing a dense subset of $\Z_p$ on which multiplication by $r$ takes a similar form algebraically as multiplication by $s$ for the standard $s$-adic integers. We then use this subset to construct a faithful representation of $A_r$ which is unitarily equivalent to a faithful representation of the Hensel-Steinitz algebras mentioned in Section \ref{Hensel}. 
\subsection{A Faithful Representation of $A_r$} Consider an integer $r$ with $|r|_p = p^{-N}$. Hence, there is $r'$ with $|r'|_p = 1$, such that $r = r' p^N$.  Denote  $\mathcal{D}^{(N)}_r$ to be the following subset of $\Z_p$: 
$$
\mathcal{D}^{(N)}_r = \left\{ x_0 + x_1r + x_2r^2 + \dots + x_n r^n : n\in\Z_{\geq 0} ,\, x_i \in \{0, \dots , p^N-1 \} \right\}.
$$
We have the following lemma. 
\begin{lem}
$\mathcal{D}^{(N)}_r$ is a dense subset of $\Z_p$. 
\end{lem}
\begin{proof}
We can construct an approximation of any $x \in \Z_p$ by elements of $\mathcal{D}^{(N)}_r$ by repeatedly using the following division algorithm:
there exist $q, c \in \Z_p$ such that:
\begin{equation}\label{div_alg}
x=qr+c
\end{equation}
with $c\in \{0, \dots , p^N-1 \}$. Thus, we only need to establish this algorithm.

To construct $q, c$ consider the expansion
$$x = \sum_{i=0}^\infty x_i p^i \in \Z_p,$$ 
with $x_i \in \{0, \dots , p-1\}$. We let $c = \sum_{i=0}^{N-1} x_i p^i$, and notice that 
$$
|x - c|_p \leq p^{-N}=|r|_p. 
$$
Consequently, $r$ divides $x-c$ in $\Z_p$ giving the quotient $q$ in the formula \eqref{div_alg}.
This completes the proof. 
\end{proof}
The above lemma allows one to construct the following faithful representation of $A_r$. 
\begin{prop}
The map $\tilde{\pi} : A_r \to B(\ell^2( \mathcal{D}^{(N)}_r))$ given by
\begin{equation*}
  \begin{aligned}
 &V_r \mapsto v_r & \text{ where }\quad &v_r E_x = E_{rx}  \\
   &M_f \mapsto M_{r,f} & \text{ where } \quad&M_{r,f} E_x = f(x) E_x 
\end{aligned}
\end{equation*}
defines a faithful representation of $A_r$. 
\end{prop}
\begin{proof}
It is clear that $\tilde{\pi}$ defines a $*$-representation of $A_r$. To show the O'Donovan conditions,  we proceed similarly to what we described in previous sections. First, we introduce a one-parameter group of unitary operators $U_r(\theta): \ell^2(\mathcal{D}_r^{(N)}) \to \ell^2(\mathcal{D}_r^{(N)})$ given by 
$$
U_r(\theta) E_{x} = e^{2 \pi i \theta \kappa(x)} E_{x},
$$
where $\kappa(x) = n$ for $x = \sum_{i=0}^n x_i r^i$ and $x_n \ne 0$. The key property of $\kappa(x)$ is:
\begin{equation*}
\kappa(rx)=\kappa(x)+1.
\end{equation*}
For $a \in C^*(v_r, M_{r,f} : f \in C(\Z_p))$, define also 
$$
\rho_\theta(a) = U_r(\theta) a U_r(-\theta),
$$
as well as an expectation
$$
\mathbb{E}(a) = \int_0^1 \rho_\theta(a) d \theta. 
$$
Clearly, we have: 
$$
\| \mathbb{E}(a) \| \leq \int_0^1 \| U_r(\theta) \| \|a \| \| U_r(-\theta) \| d \theta \leq \|a\|. 
$$
Note also that 
$$
\rho_\theta(v_r)= e^{2 \pi i \theta}v_r,
$$
while 
$$
\rho_\theta (M_{r,f}) = M_{r,f}. 
$$
Let $a$ be of the form 
$$
a = \sum_{m \geq 0} v^m_r M_{r,f_m} + \sum_{m < 0} M_{r,f_m} (v^*_r)^{-m},
$$
where the sums are finite.  A straightforward calculation shows that all terms for which $m \neq 0$ will have an expectation of $0$.
Hence, 
$$
\mathbb{E}(a) = M_{r, f_0}, 
$$
and, for $a$ as above, we have
$$
\| M_{r,f_0} \| \leq \|a\|.
$$
Since the above formula holds, by Proposition 2.1 in \cite{boyd1993faithful}, it suffices to check that the representation is faithful on the invariant part,  which is generated by $M_{r,f}$ for $f \in C(\Z_p)$. If $M_{r,f}$ vanishes, then $f$ must vanish on $\mathcal{D}^{(N)}_r \subseteq C(\Z_p)$, which is dense. It follows by continuity that $f = 0$, and the representation is faithful. 
\end{proof}
\subsection{$A_r$ as a Hensel-Steinitz Algebra with $s = p^N$} Consider the Hensel-Steinitz algebra  associated with $s = p^N$. Denote the faithful representation of $HS(s)$ on $\ell^2(\Z_{\geq 0})$ described in  Section \ref{Hensel} by $\pi_{V}: HS(s) \to B(\ell^2(\Z_{\geq 0}))$. In what follows, we construct a unitary equivalence between $\tilde{\pi}$ and  $\pi_{V}$. Consider the  map $\mathcal{U}:\ell^2(\Z_{\geq 0}) \to \ell^2(\mathcal{D}^{(N)}_r)$ defined on basis elements as follows: 
$$
\mathcal{U} E_{k(x)} = E_{x},
$$
where  $k(x) = \sum_{i=0}^n x_i p^{iN}$ and $x = \sum_{i=0}^n x_i r^i$, with $x_i \in \{0, \dots , p^N - 1 \}$.  It is easy to check this map is invertible, and satisfies 
$$
\mathcal{U} V \mathcal{U}^{-1} = v_r,
$$
where again $V$ is as defined equation \eqref{V shift} with $s = p^N$. Similarly, we have that
$$
C^*( \mathcal{U} M_f \mathcal{U}^{-1} : f \in C(\Z_{p^N}) ) = C^*(M_{r,f}) \cong C(\Z_{p^N}). 
$$
These considerations, along with the facts that $\tilde{\pi}$ is a faithful representation of $A_r$ and $\pi_{V}$ is a faithful representation of $HS(s)$, immediately lead to the following proposition.
\begin{prop}
    Let $r = r'p^N$, with $|r'|_p = 1$. We have the following isomorphism of C$^*$-algebras: 
    $$
    A_r \cong HS(s),
    $$
    where $HS(s)$ denotes the Hensel-Steinitz algebra corresponding to $s = p^N$. 
\end{prop}


\begin{thebibliography}{99}

\bibitem{boyd1993faithful}
Sarah Boyd, Navin Keswani, and Iain Raeburn.  Faithful Representations of Crossed Products by Endomorphisms. \textit{Proc. AMS}, 118(2):427--436, 1993.

\bibitem{cuntz2008c}
Joachim Cuntz. C$^*$-algebras Associated with the $ax+b$-semigroup over $\N$. \textit{K-Theory and Noncommutative Geometry}, 2:201--215, 2008.

\bibitem{exel2003new}
Ruy Exel. A New Look at the Crossed Product of a C$^*$-algebra by an Endomorphism. \textit{Erg. Theo. Dyn. Sys.}, 23(6):1733--1750, 2003.

\bibitem{exel2008new} 
Ruy Exel. A New Look at the Crossed Product of a C$^*$-algebra by a Semigroup of Endomorphisms. \textit{Erg. Theo. Dyn. Sys.}, 28(3):749--789, 2008.

\bibitem{HKMP} 
Shelley Hebert, Slawomir Klimek, Matt McBride, and J. Wilson Peoples. Noncommutative Geometry of the Hensel-Steinitz Algebra. \textit{arXiv preprint}, 2023.


\bibitem{klimek2021aspects} 
Slawomir Klimek, Matt McBride, and J. Wilson Peoples. Aspects of Noncommutative Geometry of Bunce–Deddens Algebras. \textit{Jour. Noncommut. Geom.} 17(4):1391--1423, 2023.


\bibitem{li2014k} 
Xin Li. On K-Theoretic Invariants of Semigroup C$^*$-algebras Attached to Number Fields. \textit{Advances in Mathematics}, 264:371--395, 2014.

\bibitem{paschke} 
W. Paschke. The Crossed Product of a C$^*$-algebra by an Endomorphism. \textit{Proc. AMS}, 80:113-118, 1980.

\bibitem{rordam2000introduction}
Mikael R{\o}rdam, Flemming Larson, and N Laustsen. \textit{An Introduction to K-Theory for C$^*$-algebras.} Number 49. Cambridge University Press, 2000.

\bibitem{schochet1982topological}
Claude Schochet. Topological Methods for C$^*$-algebras II. Geometry Resolutions and the K\"{u}nneth Formula. \textit{Pac. Jour. Math.}, 98(2):443--458, 1982.

\bibitem{stacey1993crossed}
Peter J Stacey. Crossed Products of C$^*$-algebras by $*$-endomorphisms. \textit{Jour. of the Aus. Math. Soc.}, 54(2):204--212, 1993.
\end{thebibliography}
\end{document}